\documentclass[11pt,a4paper]{article}
\usepackage[T1]{fontenc}
\usepackage[utf8]{inputenc}
\usepackage{lmodern}
\usepackage{amsmath,amssymb,amsthm}
\usepackage[a4paper,margin=27mm]{geometry}
\usepackage{microtype}
\usepackage{enumitem}
\usepackage[hidelinks]{hyperref}
\hypersetup{pdftitle={The Fundamental Semigroup and Associative Lifts of Topological Semigroups},
 pdfauthor={Stefano Bonzio and Andrea Loi}}
\usepackage{fancyhdr}
\newtheorem{theorem}{Theorem}[section]
\newtheorem{lemma}[theorem]{Lemma}
\newtheorem{proposition}[theorem]{Proposition}
\newtheorem{corollary}[theorem]{Corollary}
\theoremstyle{definition}
\newtheorem{definition}[theorem]{Definition}
\newtheorem{example}[theorem]{Example}
\theoremstyle{remark}
\newtheorem{remark}[theorem]{Remark}
\newcommand{\RR}{\mathbb R}
\newcommand{\ZZ}{\mathbb Z}
\newcommand{\id}{\operatorname{id}}
\newcommand{\Deck}{\operatorname{Deck}}
\newcommand{\Ad}{\operatorname{Ad}}

\title{\textbf{The Fundamental Semigroup and\\Associative Lifts of Topological Semigroups}}
\author{Stefano Bonzio \and Andrea Loi}
\date{}
\begin{document}
\maketitle

\begin{abstract}
We introduce the fundamental semigroup $\Lambda(S)$ of a nonempty
path-connected topological semigroup $S$, defined using right
translations and change-of-basepoint isomorphisms on its fundamental
group. We prove that this construction is closed under composition
and independent of the chosen basepoint up to isomorphism.
It is also invariant under homotopy
equivalences that preserve multiplication up to homotopy.
We characterize exactly when the fundamental
semigroup is a group and identify that group with an inner
automorphism group.
For a connected, locally path-connected and semilocally simply
connected topological semigroup, we show that multiplication admits
a continuous associative lift to the universal cover if and only if
$\Lambda(S)$ contains an idempotent. Every such idempotent can be
realized by an associative lift, and any two associative lifts are
isomorphic over $S$ through a deck transformation.
We obtain existence results when the fundamental group has finite
inner automorphism group or contains no subgroup isomorphic to
Thompson's group $F$, and when the underlying space is a
finite-dimensional CW complex. In contrast, we construct a
path-connected countable CW topological semigroup with cells in
arbitrarily high dimensions whose multiplication admits no
continuous associative lift to its universal cover.
\end{abstract}


\normalsize

\section{Introduction}\label{sec:introduction}

Under the standard covering-space hypotheses, the universal cover
of a connected topological group inherits a compatible group
structure. For a topological semigroup, the analogous question is
more delicate. Indeed, if $p:\widetilde S\to S$ is a universal
covering of a connected, locally path-connected and semilocally
simply connected topological semigroup, the multiplication on $S$
always admits a continuous lift to $\widetilde S$. Such a lift,
however, need not be associative: the two bracketings of a triple
product can differ by a deck transformation. The problem is thus
to determine when a multiplication lift can be chosen associative.

A classical positive result is due to Kahn
\cite[Section~2, Theorem~1]{Kahn}: if the base semigroup contains
an idempotent, its simply connected covering admits a compatible
associative multiplication with a chosen point over that
idempotent made idempotent. Kahn's comparison theorem also concerns
covering semigroups with prescribed idempotents
\cite[Section~2, Theorem~2]{Kahn}. An analogous idempotent
hypothesis appears in the semigroup-groupoid setting of Mann'a
\cite[Theorem~4.5]{Manna2013}; related constructions for Lie
semigroups are discussed in \cite{Neeb1992}. These results leave
open a natural question: can associative lifting be characterized
without assuming that $S$ itself has an idempotent?

To formulate our main result, we introduce an algebraic invariant
of a nonempty path-connected topological semigroup $S$, called its
\emph{fundamental semigroup}. Fix $a\in S$ and let
$R_a:S\to S$ denote right translation, $R_a(x)=xa$.
For any path $\alpha:a\to a^2$, let
$C_\alpha:\pi_1(S,a^2)\to\pi_1(S,a)$ be the corresponding
change-of-basepoint isomorphism. We set
\[
 \Lambda(S,a)=
 \bigl\{C_\alpha\circ (R_a)_*:\alpha:a\to a^2\bigr\}
 \subseteq\operatorname{End}(\pi_1(S,a)).
\]
We show in Section~\ref{sec:fundamental-semigroup} that this
family is closed under composition and that its isomorphism type
is independent of the basepoint $a$; we denote the resulting
semigroup, up to isomorphism, by $\Lambda(S)$. It depends on both
the topology and the multiplication on $S$ and is defined without
local covering-space hypotheses. For any semigroup $T$, we write
$E(T)=\{t\in T:t^2=t\}$ for its set of idempotents.

We use the following terminology for multiplication lifts on
universal coverings. A universal covering
$p:\widetilde S\to S$ is understood to have path-connected and
simply connected total space. Whenever a universal covering $p$
is fixed, we reserve the notation $G=\Deck(p)$ for its deck
transformation group. A \emph{continuous lift of multiplication} is
a continuous map $\mu:\widetilde S\times\widetilde S\to\widetilde S$
satisfying
\begin{equation}\label{eq:compatibility}
 p(\mu(u,v))=p(u)p(v)\qquad(u,v\in\widetilde S).
\end{equation}
It is \emph{associative} if
\begin{equation}\label{eq:associativity}
 \mu(\mu(u,v),w)=\mu(u,\mu(v,w))
 \qquad(u,v,w\in\widetilde S).
\end{equation}
In this case $p$ is a topological semigroup homomorphism.

Our principal theorem gives an exact algebraic criterion for the
existence of an associative lift and an accompanying uniqueness
statement.

\begingroup
\renewcommand{\thetheorem}{A}\renewcommand{\theHtheorem}{A}%
\begin{theorem}[Idempotent criterion for associative lifts]\label{thm:main}
Let $S$ be a nonempty connected, locally path-connected and
semilocally simply connected topological semigroup, and let
$p:\widetilde S\to S$ be its universal covering. Then:
\begin{enumerate}[label=\textup{(\alph*)},leftmargin=*]
\item The multiplication of $S$ admits a continuous associative
lift to $\widetilde S$ if and only if
\begin{equation}\label{eq:main-criterion}
 E(\Lambda(S))\ne\varnothing.
\end{equation}
\item Whenever associative lifts exist, any two of them are
isomorphic over $S$: if $\mu$ and $\mu'$ are associative lifts,
there exists a deck transformation $g\in\Deck(p)$ such that
\[
 \mu'(gu,gv)=g\mu(u,v)\qquad(u,v\in\widetilde S).
\]
\end{enumerate}
\end{theorem}
\endgroup
\addtocounter{theorem}{-1}

Unlike the earlier sufficient conditions requiring an idempotent
in $S$, Theorem~\ref{thm:main} only requires an idempotent in
$\Lambda(S)$. For instance, the semigroup
$S=(0,\infty)\times\mathbb S^1$, with
$(t,z)(s,w)=(t+s,zw)$, has no idempotent points, while
$(0,\infty)\times\RR$ carries the associative multiplication
$(t,u)(s,v)=(t+s,u+v)$ lifting that of $S$.
The proof also shows that every idempotent of $\Lambda(S)$
is realized by some associative lift, not merely that the
existence of an idempotent guarantees the existence of a lift.

The criterion yields three useful existence results.

For a group $H$, write $Z(H)$ for its center and
$\operatorname{Inn}(H)$ for its group of inner automorphisms, so that
$\operatorname{Inn}(H)\cong H/Z(H)$.

\begin{corollary}[Finite inner automorphism group]\label{cor:finite-inner}
Let $S$ satisfy the hypotheses of Theorem~\ref{thm:main}.
If $\operatorname{Inn}(\pi_1(S,a))$ is finite for some
(equivalently, every) $a\in S$, then $\Lambda(S)$ is finite
and $E(\Lambda(S))\ne\varnothing$.
Consequently, the multiplication of $S$ admits a continuous
associative lift to its universal cover, unique up to isomorphism
over $S$. In particular, this holds whenever the fundamental group
is abelian or finite.
\end{corollary}

Recall that Thompson's group $F$ admits the infinite presentation
\[
 F=\bigl\langle x_0,x_1,x_2,\ldots\ \bigm|\
 x_i^{-1}x_jx_i=x_{j+1}\quad (0\le i<j)\bigr\rangle.
\]
The universal group associated with conjugacy idempotents
in Freyd--Heller \cite[Section~4]{FreydHeller1993}
is isomorphic to this group; see also
\cite{FioreLeinster} for a categorical characterization of $F$.

\begin{corollary}[Fundamental groups without Thompson's $F$]
\label{cor:thompson-free}
Let $S$ satisfy the hypotheses of Theorem~\ref{thm:main}.
If $\pi_1(S,a)$ contains no subgroup isomorphic to Thompson's
group $F$, then $E(\Lambda(S))\ne\varnothing$. Consequently,
the multiplication of $S$ admits a continuous associative lift
to its universal cover, unique up to isomorphism over $S$.
This holds in particular when the fundamental group is
solvable, residually finite, torsion, or linear over a field.
In view of the recent nonamenability result for
Thompson's group $F$ \cite{OpenAI2026Nonamenable,KimLodha2026},
the conclusion also holds when the fundamental group is amenable.
\end{corollary}

Examples of nonabelian groups covered by this criterion include
nonabelian free groups, finitely generated nilpotent groups,
and linear groups such as $SL_2(\mathbb Z)$.

Recall that a CW complex is \emph{finite-dimensional} if its cell
dimensions are bounded, although it may have infinitely many cells.
Every connected CW complex is path-connected, locally path-connected,
and semilocally simply connected, and hence admits a universal cover
\cite[Section~1.3 and Appendix A]{Hatcher}.

\begin{corollary}[Finite-dimensional CW semigroups]\label{cor:CW}
Let $S$ be a nonempty connected topological semigroup whose
underlying space is a finite-dimensional CW complex. Then
$E(\Lambda(S))\ne\varnothing$. Thus $S$ admits a continuous
associative lift to its universal cover, unique up to isomorphism
over $S$.
\end{corollary}

It is worth noting that finite dimensionality does not prevent
the fundamental group of a CW complex from containing Thompson's
group $F$. Indeed, $F$ is finitely presented (in fact, it admits
a presentation with two generators and two relations; see
\cite[Section~3]{CannonFloydParry1996}). Hence there exists a finite
connected two-dimensional CW complex $X$ with $\pi_1(X)\cong F$.
Endowing $X$ with the left-zero multiplication $xy=x$, we obtain
a finite-dimensional CW topological semigroup with fundamental
group $F$. Thus Corollary~\ref{cor:CW} does not follow from
the group-theoretic hypothesis in Corollary~\ref{cor:thompson-free}.

These positive results have a sharp limitation in the
infinite-dimensional setting.

\begingroup
\renewcommand{\thetheorem}{B}\renewcommand{\theHtheorem}{B}%
\begin{theorem}[An infinite-dimensional obstruction]\label{thm:counterexample}
There is a nonempty path-connected countable CW topological
semigroup $S_\infty$ with cells in arbitrarily high dimensions and
$\pi_1(S_\infty)\cong F$, Thompson's nonabelian group, such that
\[
 E(\Lambda(S_\infty))=\varnothing.
\]
Equivalently, its multiplication has continuous lifts to the
universal cover, but no associative lift.
\end{theorem}
\endgroup
\addtocounter{theorem}{-1}

\begin{remark}[The compact case]\label{rem:compact}
Every nonempty compact Hausdorff topological semigroup has an
idempotent by the Ellis--Numakura lemma \cite{Numakura,Ellis}.
Consequently, if such a semigroup also meets the covering-space
hypotheses, Kahn's theorem \cite[Section~2, Theorem~1]{Kahn}
provides an associative lift. A finite CW complex is compact,
but a finite-dimensional CW complex need not be, so this does
not subsume Corollary~\ref{cor:CW}.
\end{remark}

\begin{remark}
The example in Theorem~\ref{thm:counterexample} has no
idempotent points in $S_\infty$: otherwise Kahn's theorem would
produce an associative lift.
Moreover, Corollary~\ref{cor:thompson-free} (see the necessary
condition~\eqref{eq:freyd-heller-obstruction} in
Section~\ref{sec:existence-criteria}) shows that any such
counterexample must have a fundamental group containing a
subgroup isomorphic to Thompson's group $F$. Containing $F$ is,
however, not sufficient to preclude associative lifts: the
finite two-dimensional CW complex $X$ with $\pi_1(X)\cong F$,
endowed with the left-zero multiplication $xy=x$, consists of
idempotents and hence admits an associative lift.
Theorem~\ref{thm:counterexample} realizes the group-theoretic
obstruction with $\pi_1(S_\infty)\cong F$ itself.
\end{remark}

Our approach is based on the semigroup $\mathcal L_m(p)$ of deck
endomorphisms associated with continuous multiplication lifts.
We show that $\mathcal L_m(p)$ is a nonempty subsemigroup of
$\operatorname{End}(G)$ and that its idempotents are precisely
the endomorphisms induced by associative lifts. This is proved
by analyzing and normalizing the associativity defect of a lift.
We then identify $\mathcal L_m(p)$ with the intrinsically defined
fundamental semigroup $\Lambda(S)$, thereby obtaining the
criterion in Theorem~\ref{thm:main}. The algebraic normalization
is related to the theory of idempotents modulo inner automorphisms
of Freyd--Heller \cite{FreydHeller1993}. Their results also explain
the role of Thompson's group $F$: an obstruction to finding a
strictly idempotent representative in the relevant inner orbit
forces a subgroup isomorphic to $F$. The finite-dimensional CW
result uses the homotopy-splitting theorem of Hastings--Heller
\cite{HastingsHeller}, whereas the infinite-dimensional
counterexample is constructed using the free strict monoidal
groupoid of Fiore--Leinster \cite{FioreLeinster}.

The paper is organized as follows.
Section~\ref{sec:fundamental-semigroup} gives the detailed
construction of $\Lambda(S)$, proves basepoint independence and
invariance under multiplicative homotopy
equivalences (hence under topological semigroup isomorphisms),
and establishes a basic property of its idempotents.
We also characterize when $\Lambda(S)$ is a group
in terms of normality of the image of a right translation on
fundamental groups.
Section~\ref{sec:defect} analyzes the associativity defect
and its normalization, and proves that $\mathcal L_m(p)$ is a
semigroup whose idempotents are precisely those induced by
associative lifts; it also distinguishes these endomorphisms
from the lifts inducing them and provides a non-associative example.
Section~\ref{sec:proof-A} identifies $\Lambda(S)$ with
$\mathcal L_m(p)$ and deduces Theorem~\ref{thm:main}, including
the uniqueness statement for associative lifts.
Section~\ref{sec:existence-criteria} proves the three
existence corollaries, and Section~\ref{sec:proof-B} constructs
the counterexample of Theorem~\ref{thm:counterexample}.

\section{The Fundamental Semigroup}
\label{sec:fundamental-semigroup}

We develop the fundamental semigroup introduced in the
Introduction and prove its basic properties. The arguments require
only path-connectedness, not the existence of a universal covering.

Throughout, homotopies need not preserve basepoints. By a
\emph{topological semigroup} we mean a nonempty Hausdorff space $S$
with a jointly continuous associative multiplication
$m:S\times S\to S$, written $xy=m(x,y)$. We do not assume an
identity, commutativity, or idempotent elements. We write
$E(T)=\{t\in T:t^2=t\}$ for the idempotents of a semigroup $T$.

\begin{definition}[Fundamental semigroup]\label{def:fundamental-semigroup}
Let $S$ be a nonempty path-connected topological semigroup.
Fix $a\in S$ and let $R_a:S\to S$ denote right
translation by $a$, namely
$R_a(x)=xa$. For every path $\alpha:a\to a^2$, let
\[
 C_\alpha:\pi_1(S,a^2)\longrightarrow\pi_1(S,a),\qquad
 C_\alpha([\ell])=[\alpha*\ell*\overline\alpha]
\]
be the change-of-basepoint isomorphism (paths are concatenated
from left to right), and put
\[
 \theta_\alpha:=C_\alpha\circ (R_a)_*
 \in\operatorname{End}(\pi_1(S,a)).
\]
With $S$ fixed, the initial point of $\alpha$ determines $a$ and
hence $R_a$. We therefore write $\theta_\alpha$ for these
right-translation endomorphisms throughout.
Define the \emph{fundamental semigroup} of $S$ \emph{based at $a$} by
\begin{equation}\label{eq:lambda-definition}
 \Lambda(S,a):=
 \bigl\{\theta_\alpha:\alpha\text{ is a path from }a
 \text{ to }a^2\bigr\}
 \subseteq\operatorname{End}(\pi_1(S,a)),
\end{equation}
with multiplication given by composition of endomorphisms.
Closure under composition is verified below.
\end{definition}

Fix a reference path $\alpha_0:a\to a^2$. For any connecting path
$\alpha:a\to a^2$, the change-of-path formula is
\[
 \theta_\alpha
 =\Ad_{[\alpha*\overline{\alpha_0}]}\circ \theta_{\alpha_0},
\]
where $\Ad_g(h)=ghg^{-1}$. Conversely, every inner translate of
$\theta_{\alpha_0}$ is obtained by prepending a loop at $a$ to
$\alpha_0$. Therefore
\begin{equation}\label{eq:lambda-inner-orbit}
 \Lambda(S,a)=
 \bigl\{\Ad_g\circ \theta_{\alpha_0}:
 g\in\pi_1(S,a)\bigr\}.
\end{equation}
The reference path $\alpha_0$ is fixed for this description, and the
resulting orbit is independent of its choice. Equivalently, any
element of $\Lambda(S,a)$ may be used as the orbit representative.
This description will also be used in the proofs of
Corollaries~\ref{cor:finite-inner}, \ref{cor:thompson-free}
and~\ref{cor:CW}.

\medskip
\noindent\textbf{Standard change-of-basepoint identities.}
Both formulas are classical and are stated without proof. The
change-of-basepoint convention used throughout agrees with
\cite[Section~1.1, Proposition~1.5]{Hatcher}.
If $f:X\to Y$ is continuous and $\alpha:x_0\to x_1$ is a path in
$X$, naturality of change of basepoint gives
\begin{equation}\label{eq:standard-naturality}
 f_{*,x_0}\circ C_\alpha
 =C_{f\circ \alpha}\circ f_{*,x_1}
\end{equation}
(see \cite[Section~1.1, Exercise~15]{Hatcher}).
If $f_0,f_1:X\to Y$ are homotopic by
$H:X\times[0,1]\to Y$, and $h(t)=H(x_0,t)$ is the path
traced by the basepoint, then
\begin{equation}\label{eq:standard-homotopy}
 (f_0)_{*,x_0}=C_h\circ (f_1)_{*,x_0}
\end{equation}
(see \cite[Section~1.1, Lemma~1.19]{Hatcher}), where
$C_h:\pi_1(Y,f_1(x_0))\to\pi_1(Y,f_0(x_0))$.
The subscripts on induced homomorphisms specify their source
basepoints and will be omitted when these are clear.
\medskip

\begin{proof}[Closure under composition]
Let $\alpha,\beta:a\to a^2$ be paths and write
$L_a(x)=ax$. Applying~\eqref{eq:standard-naturality} gives
\[
 \theta_\alpha\circ \theta_\beta
 =C_{\alpha*(R_a\circ \beta)}\circ (R_a^2)_*.
\]
The homotopy $H(x,t)=x\beta(t)$ runs from $R_a$ to
$R_a^2=R_{a^2}$, and its path at the basepoint $a$ is
$L_a\circ \beta:a^2\to a^3$. The homotopy
identity~\eqref{eq:standard-homotopy} therefore implies
\[
 C_{L_a\circ \beta}\circ (R_a^2)_*=(R_a)_*.
\]
Consequently,
\begin{equation}\label{eq:lambda-composition-path}
 \begin{aligned}
 \theta_\alpha\circ \theta_\beta
 &=C_{\alpha*(R_a\circ \beta)*
       \overline{(L_a\circ \beta)}}\circ (R_a)_*\\
 &=\theta_\gamma,
 \end{aligned}
\end{equation}
where
\[
 \gamma:=\alpha*(R_a\circ \beta)*
          \overline{(L_a\circ \beta)}:a\to a^2.
\]
Thus $\Lambda(S,a)$ is closed under composition. It is nonempty
because $S$ is path-connected, and associativity is inherited from
$\operatorname{End}(\pi_1(S,a))$.
\end{proof}

\begin{lemma}[Basepoint independence]\label{lem:basepoint-independence}
For any $a,b\in S$ and any path $\eta:a\to b$, the
change-of-basepoint isomorphism
$C_\eta:\pi_1(S,b)\xrightarrow{\cong}\pi_1(S,a)$ satisfies
\begin{equation}\label{eq:lambda-basepoint-independence}
 C_\eta\circ \Lambda(S,b)\circ C_\eta^{-1}=\Lambda(S,a).
\end{equation}
In particular, $\Lambda(S,a)\cong\Lambda(S,b)$ as semigroups.
\end{lemma}

\begin{proof}
Let $\beta:b\to b^2$ be any path, and put
$L_a(x)=ax$. Define the paths
\[
 \delta:=\eta*\beta*(R_b\circ \overline\eta):a\to ab,
 \qquad h:=L_a\circ \eta:a^2\to ab,
\]
and $\alpha:=\delta*\overline h:a\to a^2$.
For a loop $u$ at $a$, the definitions and the
naturality identity~\eqref{eq:standard-naturality} give
\[
 (C_\eta\circ \theta_\beta\circ C_\eta^{-1})([u])
 = [\delta*(R_b\circ u)*\overline\delta].
\]
The homotopy $H(x,t)=x\eta(t)$ runs from $R_a$ to $R_b$.
Its trace at the basepoint $a$ is $h$, so the homotopy
identity~\eqref{eq:standard-homotopy} yields
\[
 [R_b\circ u]=[\overline h*(R_a\circ u)*h]
 \quad\text{in }\pi_1(S,ab).
\]
Consequently,
\[
 \begin{aligned}
 (C_\eta\circ \theta_\beta\circ C_\eta^{-1})([u])
 &= [\alpha*(R_a\circ u)*\overline\alpha]\\
 &=\theta_\alpha([u]).
 \end{aligned}
\]
Thus conjugation by $C_\eta$ maps $\Lambda(S,b)$ into
$\Lambda(S,a)$. Applying the same argument to the reversed path
$\overline\eta:b\to a$ gives the reverse inclusion, since
$C_{\overline\eta}=C_\eta^{-1}$. Hence
\eqref{eq:lambda-basepoint-independence} holds, and conjugation
by $C_\eta$ is a semigroup isomorphism.
\end{proof}

By Lemma~\ref{lem:basepoint-independence}, we write $\Lambda(S)$
for the fundamental semigroup considered up to isomorphism,
and retain $\Lambda(S,a)$ when working with its realization in
$\operatorname{End}(\pi_1(S,a))$.

\begin{theorem}[Invariance under multiplicative homotopy equivalences]
\label{thm:fundamental-multiplicative-homotopy-invariance}
Let $(S,m_S)$ and $(T,m_T)$ be nonempty path-connected topological
semigroups. Suppose that $f:S\to T$ is a homotopy equivalence of
the underlying spaces and that
\[
 f\circ m_S\simeq m_T\circ (f\times f)
\]
as continuous maps $S\times S\to T$. Then
\[
 \Lambda(S)\cong\Lambda(T).
\]
More precisely, fix a homotopy
$H:S\times S\times[0,1]\to T$ satisfying
\[
 H(x,y,0)=f(xy),\qquad H(x,y,1)=f(x)f(y).
\]
For $a\in S$, put $b=f(a)$ and let
$h(t)=H(a,a,t)$, a path from $f(a^2)$ to $b^2$.
Then for every path $\alpha:a\to a^2$,
\begin{equation}\label{eq:lambda-multiplicative-homotopy-formula}
 f_{*,a}\circ \theta^S_\alpha\circ f_{*,a}^{-1}
   =\theta^T_{(f\circ \alpha)*h},
\end{equation}
where the superscripts indicate the ambient semigroup. In
particular,
\begin{equation}\label{eq:lambda-multiplicative-homotopy-invariance}
 f_{*,a}\circ \Lambda(S,a)\circ f_{*,a}^{-1}
     =\Lambda(T,b).
\end{equation}
\end{theorem}

\begin{proof}
The homotopy $K(x,t):=H(x,a,t)$ runs from
$f\circ R_a^S$ to $R_b^T\circ f$. Its trace at $a$
is $h(t)$. The homotopy identity
\eqref{eq:standard-homotopy} therefore gives
\[
 f_{*,a^2}\circ (R_a^S)_*
   = C_h\circ (R_b^T)_*\circ f_{*,a}.
\]
For a path $\alpha:a\to a^2$, naturality
\eqref{eq:standard-naturality} gives
\[
 f_{*,a}\circ C_\alpha^S
   =C_{f\circ \alpha}^T\circ f_{*,a^2}.
\]
Combining the two identities, and using
$C_{f\circ \alpha}^T\circ C_h=C_{(f\circ \alpha)*h}^T$,
we obtain
\[
 \begin{aligned}
 f_{*,a}\circ \theta^S_\alpha\circ f_{*,a}^{-1}
 &=C_{f\circ \alpha}^T\circ C_h\circ (R_b^T)_*\\
 &=\theta^T_{(f\circ \alpha)*h},
 \end{aligned}
\]
proving~\eqref{eq:lambda-multiplicative-homotopy-formula}.

It remains to prove that conjugation by $f_{*,a}$ maps the
entire fundamental semigroup onto $\Lambda(T,b)$.
Choose a reference path $\alpha_0:a\to a^2$ and put
$\rho=(f\circ \alpha_0)*h:b\to b^2$.
By the inner-orbit formula~\eqref{eq:lambda-inner-orbit},
\[
 \Lambda(S,a)
 =\{\Ad_g\circ \theta^S_{\alpha_0}:
                        g\in\pi_1(S,a)\}.
\]
Since $f$ is a homotopy equivalence,
$f_{*,a}:\pi_1(S,a)\to\pi_1(T,b)$ is an isomorphism.
Thus~\eqref{eq:lambda-multiplicative-homotopy-formula}
implies
\[
 \begin{aligned}
 f_{*,a}\circ \Lambda(S,a)\circ f_{*,a}^{-1}
 &=\{\Ad_{f_{*,a}(g)}\circ \theta^T_\rho:
                              g\in\pi_1(S,a)\}\\
 &=\{\Ad_k\circ \theta^T_\rho:
                              k\in\pi_1(T,b)\}\\
 &=\Lambda(T,b),
 \end{aligned}
\]
where the last equality again uses
\eqref{eq:lambda-inner-orbit}. Conjugation by the group
isomorphism $f_{*,a}$ preserves composition and therefore gives
the asserted semigroup isomorphism.
\end{proof}

\begin{corollary}[Invariance under semigroup isomorphisms]
\label{prop:fundamental-isomorphism-invariance}
If $f:S\to T$ is an isomorphism of nonempty path-connected
topological semigroups, then
$\Lambda(S)\cong\Lambda(T)$. More precisely, for every $a\in S$,
the induced isomorphism $f_{*,a}$ conjugates $\Lambda(S,a)$ onto
$\Lambda(T,f(a))$.
\end{corollary}

\begin{proof}
A topological semigroup isomorphism is a homotopy equivalence
and satisfies $f\circ m_S=m_T\circ (f\times f)$.
Apply Theorem~\ref{thm:fundamental-multiplicative-homotopy-invariance}
with the constant homotopy; its basepoint trace $h$ is constant,
and~\eqref{eq:lambda-multiplicative-homotopy-formula} reduces to
\[
 f_{*,a}\circ \theta^S_\alpha\circ f_{*,a}^{-1}
 =\theta^T_{f\circ \alpha}.
\]
\end{proof}

\begin{corollary}[Homotopic associative multiplications]
\label{cor:fundamental-homotopic-multiplications}
Let $X$ be a nonempty path-connected Hausdorff space, and let
$m_0,m_1:X\times X\to X$ be continuous associative
multiplications. If $m_0\simeq m_1$ as maps $X\times X\to X$,
then for every $a\in X$,
\[
 \Lambda((X,m_0),a)=\Lambda((X,m_1),a)
 \quad\text{inside }\operatorname{End}(\pi_1(X,a)).
\]
\end{corollary}

\begin{proof}
Apply Theorem~\ref{thm:fundamental-multiplicative-homotopy-invariance}
to the identity map $f=\id_X:(X,m_0)\to(X,m_1)$.
Its induced map on $\pi_1(X,a)$ is the identity, so
\eqref{eq:lambda-multiplicative-homotopy-invariance}
gives equality, rather than merely an abstract isomorphism.
\end{proof}

\begin{remark}[Not an invariant of the underlying homotopy type]
\label{rem:fundamental-not-ordinary-homotopy-invariant}
The multiplicative compatibility hypothesis is essential:
$\Lambda(S)$ is not, in general, determined by the homotopy
type of the underlying space. Indeed, on the bouquet of two
circles, the multiplication in
Example~\ref{ex:commutative-infinite-idempotents} gives a
countably infinite fundamental semigroup. On the same underlying
space, the right-zero multiplication $x*y=y$ makes every right
translation constant and hence produces a singleton fundamental
semigroup. These two semigroups have identical underlying spaces
but nonisomorphic fundamental semigroups.
\end{remark}

The idempotents of the fundamental semigroup satisfy a simple
structural property.

\begin{proposition}[Idempotents are right identities]
\label{prop:fundamental-idempotents-right-identities}
If $S$ is a nonempty path-connected topological semigroup and
$E(\Lambda(S))\ne\varnothing$, then every idempotent of
$\Lambda(S)$ is a right identity for $\Lambda(S)$. In particular,
the idempotents form a left-zero semigroup:
\[
 xy=x\qquad(x,y\in E(\Lambda(S))).
\]
\end{proposition}

\begin{proof}
If $E(\Lambda(S,a))$ is nonempty, every idempotent is a right identity
for all of $\Lambda(S,a)$. Indeed, choose an idempotent
$e\in\Lambda(S,a)$. By the inner-orbit
formula \eqref{eq:lambda-inner-orbit},
\[
 \Lambda(S,a)=\{e_g:=\Ad_g\circ e:g\in\pi_1(S,a)\}.
\]
Since $e$ is a homomorphism, for every $h\in\pi_1(S,a)$ we have
\[
 e\circ \Ad_h=\Ad_{e(h)}\circ e.
\]
Using $e^2=e$, we therefore obtain, for any $g,h\in\pi_1(S,a)$,
\begin{equation}\label{eq:fundamental-idempotent-product}
 \begin{aligned}
 e_g\circ e_h
 &=\Ad_g\circ e\circ \Ad_h\circ e\\
 &=\Ad_g\circ \Ad_{e(h)}\circ e^2\\
 &=\Ad_{g e(h)}\circ e
 =e_{g e(h)}.
 \end{aligned}
\end{equation}
If $e_h$ is idempotent, then $e_{h e(h)}=e_h$.
Writing this equality in terms of inner automorphisms gives
\[
 \Ad_h\circ \Ad_{e(h)}\circ e
 =\Ad_{h e(h)}\circ e
 =\Ad_h\circ e.
\]
Since $\Ad_h$ is an automorphism, composing both sides on the
left with $\Ad_{h^{-1}}$ yields
\[
 \Ad_{e(h)}\circ e=e.
\]
Consequently, for every $g\in\pi_1(S,a)$,
\[
 \begin{aligned}
 e_g\circ e_h
 &=e_{g e(h)}\\
 &=\Ad_g\circ \bigl(\Ad_{e(h)}\circ e\bigr)\\
 &=\Ad_g\circ e=e_g.
 \end{aligned}
\]
In particular, the idempotents form a \emph{left-zero semigroup}:
\[
 xy=x\qquad(x,y\in E(\Lambda(S,a))).
\]
\end{proof}

\begin{theorem}[When the fundamental semigroup is a group]
\label{thm:fundamental-semigroup-group-criterion}
Let $S$ be a nonempty path-connected topological semigroup and fix
$a\in S$. For any path $\alpha:a\to a^2$, the following conditions
are equivalent:
\begin{enumerate}[label=(\roman*),leftmargin=2.5em]
 \item $\Lambda(S,a)$ is a group under composition;
 \item $\operatorname{Im}(\theta_\alpha)$ is a normal subgroup of
 $\pi_1(S,a)$;
 \item $\operatorname{Im}(R_a)_*$ is a normal subgroup of
 $\pi_1(S,a^2)$.
\end{enumerate}
In particular,
\begin{equation}\label{eq:fundamental-group-normality-criterion}
 \Lambda(S)\text{ is a group}
 \quad\Longleftrightarrow\quad
 \operatorname{Im}(R_a)_*\trianglelefteq\pi_1(S,a^2).
\end{equation}
If these conditions hold, $\Lambda(S,a)$ has a unique idempotent
$e$, which is its identity, and
\[
 \operatorname{Im}(e)=\operatorname{Im}(\theta_\alpha),
 \qquad
 \Lambda(S,a)\cong
 \operatorname{Inn}\bigl(\operatorname{Im}(\theta_\alpha)\bigr).
\]
In particular, this criterion is independent of the path $\alpha$
and of the basepoint $a$.
\end{theorem}

\begin{proof}
Suppose that $\Lambda(S,a)$ is a group, with identity $e$.
Then $e^2=e$. For every $g\in\pi_1(S,a)$, the inner-orbit
formula~\eqref{eq:lambda-inner-orbit} gives
$\Ad_g\circ e\in\Lambda(S,a)$, and hence
\[
 e\circ (\Ad_g\circ e)=\Ad_g\circ e.
\]
Thus $e$ fixes every element of
$g\operatorname{Im}(e)g^{-1}$. Since the fixed points of an
idempotent endomorphism are exactly its image, it follows that
\[
 g\operatorname{Im}(e)g^{-1}
 \subseteq\operatorname{Im}(e)
 \qquad(g\in\pi_1(S,a)).
\]
Applying the same inclusion to $g^{-1}$ shows that
$\operatorname{Im}(e)\trianglelefteq\pi_1(S,a)$.
Since $\theta_\alpha=\Ad_u\circ e$ for some
$u\in\pi_1(S,a)$, we also have
$\operatorname{Im}(\theta_\alpha)=\operatorname{Im}(e)$.
This proves (ii).

Conversely, suppose that
$\operatorname{Im}(\theta_\alpha)\trianglelefteq\pi_1(S,a)$.
Closure under composition and the inner-orbit formula give an
element $t\in\pi_1(S,a)$ such that
\[
 \theta_\alpha^2=\Ad_t\circ \theta_\alpha.
\]
Consequently, the restriction of $\theta_\alpha$ to its image
is conjugation by $t$. Define
\[
 e:=\Ad_{t^{-1}}\circ \theta_\alpha\in\Lambda(S,a).
\]
Normality implies
$\operatorname{Im}(e)=\operatorname{Im}(\theta_\alpha)$,
and $e$ restricts to the identity on its image. Thus $e^2=e$.
Both $\operatorname{Im}(e)$ and $\ker(e)$ are normal in
$\pi_1(S,a)$, and their intersection is trivial. Every
element $g\in\pi_1(S,a)$ factors as
\[
 g=e(g)\bigl(e(g)^{-1}g\bigr),\qquad
 e(g)\in\operatorname{Im}(e),\quad
 e(g)^{-1}g\in\ker(e).
\]
Hence these two normal subgroups centralize each other and
\begin{equation}\label{eq:fundamental-group-normal-orbit}
 \begin{aligned}
 \Lambda(S,a)
 &=\{\Ad_g\circ e:g\in\pi_1(S,a)\}\\
 &=\{\Ad_h\circ e:h\in\operatorname{Im}(e)\}.
 \end{aligned}
\end{equation}
For any $h,k\in\operatorname{Im}(e)$, since
$e|_{\operatorname{Im}(e)}=\id_{\operatorname{Im}(e)}$,
\[
 (\Ad_h\circ e)\circ (\Ad_k\circ e)=\Ad_{hk}\circ e.
\]
Therefore $e$ is the identity of $\Lambda(S,a)$, and the
inverse of $\Ad_h\circ e$ is $\Ad_{h^{-1}}\circ e$.
Restriction to $\operatorname{Im}(e)$ yields an isomorphism
\[
 \Lambda(S,a)\xrightarrow{\ \cong\ }
 \operatorname{Inn}\bigl(\operatorname{Im}(e)\bigr)
 =\operatorname{Inn}\bigl(\operatorname{Im}(\theta_\alpha)\bigr).
\]
This proves (i) and identifies the group structure.

Finally, the change-of-basepoint map
$C_\alpha:\pi_1(S,a^2)\to\pi_1(S,a)$ is an isomorphism,
and
\[
 \operatorname{Im}(\theta_\alpha)
 =C_\alpha\bigl(\operatorname{Im}(R_a)_*\bigr).
\]
Normality is preserved and reflected by isomorphisms, so
(ii) and (iii) are equivalent. The independence of the basepoint
follows from Lemma~\ref{lem:basepoint-independence}.
\end{proof}

\begin{corollary}[Direct-product splitting]
\label{cor:fundamental-group-direct-product-splitting}
If $\Lambda(S,a)$ is a group and $e$ is its identity, then
\[
 \pi_1(S,a)\cong\operatorname{Im}(e)\times\ker(e),
 \qquad
 \Lambda(S,a)\cong
 \operatorname{Im}(e)/Z\bigl(\operatorname{Im}(e)\bigr).
\]
\end{corollary}

\begin{proof}
The first isomorphism is the internal direct-product splitting
constructed in the proof of
Theorem~\ref{thm:fundamental-semigroup-group-criterion}.
The second follows from the natural isomorphism between a group
modulo its center and its inner automorphism group.
\end{proof}

\begin{example}[A nonabelian fundamental group of semigroup type]
\label{ex:fundamental-semigroup-is-nonabelian-group}
Let $X=\mathbb S^1\vee\mathbb S^1$ with the left-zero
multiplication $xy=x$, and let $o$ be the wedge point. Then
$R_o=\id_X$ and $o^2=o$, so the constant path at $o$ gives
$\theta_{c_o}=\id_{\pi_1(X,o)}$. The inner-orbit formula yields
\[
 \begin{aligned}
 \Lambda(X,o)
 &=\{\Ad_g:g\in\pi_1(X,o)\}\\
 &=\operatorname{Inn}(\pi_1(X,o))
 \cong\pi_1(X,o).
 \end{aligned}
\]
Indeed, $\pi_1(X,o)$ is a free group of rank two and has
trivial center. Thus the fundamental semigroup can be a
nontrivial, nonabelian group. In contrast, the right-zero
multiplication $xy=y$ on the same space gives the one-element
group $\Lambda(X)=\{0\}$.
\end{example}

\begin{remark}[Abelian fundamental groups]
\label{rem:abelian-fundamental-semigroup}
If $\pi_1(S,a)$ is abelian, the inner-orbit formula
\eqref{eq:lambda-inner-orbit} shows that $\Lambda(S,a)$ is a
singleton, say $\Lambda(S,a)=\{e\}$. Closure under composition
then gives $e^2=e$. Thus, as an abstract semigroup, $\Lambda(S)$
is always the one-element semigroup in this case. Its unique
element, viewed as an endomorphism of $\pi_1(S,a)$, need not be the
identity, as the following example illustrates.
\end{remark}

\begin{example}[The circle]
\label{ex:circle-fundamental-semigroup}
Let $S=\mathbb S^1$ and choose the basepoint $a=1$, so that
$\pi_1(S,1)\cong\ZZ$. For the usual multiplication $z\cdot w=zw$,
right translation by $1$ is the identity, and hence
\[
 \Lambda(S,1)=\{\id_{\ZZ}\}.
\]
For the right-zero multiplication $z\cdot w=w$, the map $R_1$
is constant and
\[
 \Lambda(S,1)=\{0\},
\]
where $0$ denotes the zero endomorphism of $\ZZ$. For the left-zero
multiplication $z\cdot w=z$, one again has
$\Lambda(S,1)=\{\id_{\ZZ}\}$. Every point of $S$ is idempotent
for either zero multiplication, although its fundamental semigroup
has only one idempotent. More generally, an endomorphism
$n\mapsto kn$ of $\ZZ$ is idempotent precisely when $k^2=k$,
so $0$ and $\id_{\ZZ}$ are the only possible endomorphisms
representing a fundamental semigroup on the circle.

Thus $\{0\}$ and $\{\id_{\ZZ}\}$ are distinct subsemigroups of
$\operatorname{End}(\ZZ)$ that are isomorphic as abstract semigroups.
\end{example}

\begin{example}[Infinitely many fundamental idempotents]
\label{ex:commutative-infinite-idempotents}
Let $S=C_u\vee C_v$ be the wedge of two circles, with common
basepoint $o$. Identify $C_u$ with the group $\mathbb S^1$, taking
$o$ as its identity. Let $i:C_u\hookrightarrow S$ be the inclusion
and let $r:S\to C_u$ be the retraction that collapses $C_v$ to $o$.
Define
\[
 x\cdot y=i\bigl(r(x)r(y)\bigr).
\]
This multiplication is continuous because $r$, $i$, and the
multiplication on $C_u$ are continuous. It is commutative because
$C_u$ is abelian. Since $r\circ i=\id_{C_u}$, we have
\[
 (x\cdot y)\cdot z
 =i\bigl(r(x)r(y)r(z)\bigr)
 =x\cdot(y\cdot z).
\]
Thus $S$ is a commutative topological semigroup whose underlying
space is a finite connected one-dimensional CW complex.

The fundamental group $\pi_1(S,o)$ is free on the two circle loops
$u$ and $v$. Since $r(o)=1$ and $i(1)=o$, we have $o^2=o$ and
$R_o=i\circ r$. We may therefore choose the constant connecting
path $c_o(t)=o$ from $o$ to $o^2$. Its change-of-basepoint
isomorphism is $C_{c_o}=\id_{\pi_1(S,o)}$, so this choice gives
\[
 e:=C_{c_o}\circ (R_o)_*
 =(R_o)_*=(i\circ r)_*\in\Lambda(S,o).
\]
The map $i\circ r$ fixes $C_u$ pointwise and collapses $C_v$ to $o$.
Hence
\[
 e(u)=u,\qquad e(v)=1,
\]
where $1$ denotes the identity element of $\pi_1(S,o)$.
In particular, $e(\pi_1(S,o))=\langle u\rangle$, and $e$ restricts
to the identity on this subgroup. Thus $e^2=e$.
By \eqref{eq:lambda-inner-orbit},
\[
 \Lambda(S,o)=\{e_g:=\Ad_g\circ e:g\in\pi_1(S,o)\},
 \qquad
 e_g(u)=gug^{-1},\quad e_g(v)=1.
\]
For every $h\in\pi_1(S,o)$, the element $e(h)$ belongs to the cyclic group
$e(\pi_1(S,o))=\langle u\rangle$ and therefore centralizes $e(\pi_1(S,o))$.
Consequently, \eqref{eq:fundamental-idempotent-product} yields
\[
 e_g\circ e_h=\Ad_{g e(h)}\circ e=\Ad_g\circ e=e_g
 \qquad(g,h\in\pi_1(S,o)).
\]
In particular, every element of $\Lambda(S,o)$ is idempotent.
The elements $e_{v^n}$, $n\in\ZZ$, are pairwise distinct, since
\[
 e_{v^n}(u)=v^nuv^{-n},\qquad e_{v^n}(v)=1,
\]
and the words $v^nuv^{-n}$ are distinct in the free group.
It follows that $\Lambda(S,o)=E(\Lambda(S,o))$ is a countably
infinite left-zero semigroup.

This example also shows that commutativity of $S$ does not force
$\pi_1(S,o)$ to be abelian or $\Lambda(S,o)$ to be a singleton.
The realization result proved below applies to every one of these
idempotents, and Theorem~\ref{thm:main}(b) ensures that the
corresponding associative lifts are all isomorphic over $S$.
\end{example}

\begin{remark}[Pointwise-convergence topology]
\label{rem:fundamental-topology}
Give $\pi_1(S,a)$ the discrete topology. On
$\operatorname{End}(\pi_1(S,a))$ use the topology of pointwise
convergence, as a subspace of $\pi_1(S,a)^{\pi_1(S,a)}$.
Composition is continuous, so $\Lambda(S,a)$
is a Hausdorff topological semigroup. It is zero-dimensional and
totally disconnected; in particular it is connected only when it
is a singleton. If $\pi_1(S,a)$ is finitely generated, this topology on
$\Lambda(S,a)$ is discrete. These assertions are also independent
of the existence of a universal covering.
\end{remark}

All the results of this section require only path-connectedness,
not the local hypotheses needed for universal covering spaces.
We now turn to the associative lifting problem.

\section{Lifts and the Associativity Defect}\label{sec:defect}

Throughout this section, $S$ satisfies the hypotheses of
Theorem~\ref{thm:main}, and $p:\widetilde S\to S$ denotes its universal
cover. We describe the deck data of a multiplication lift, the defect
of associativity, and a correction that removes the defect when an
associated endomorphism is idempotent.

\subsection{Lifts and deck transformations}

We use the standard lifting properties of universal coverings
\cite[Section~1.3]{Hatcher}. A continuous map from a path-connected,
locally path-connected, simply connected space lifts with any
prescribed value above one point. Two lifts from a connected domain
that agree at one point agree everywhere.
Write
\[
 G=\Deck(p)
 =\{g:\widetilde S\to\widetilde S\text{ a homeomorphism}:
       p\circ g=p\}.
\]
The group law is composition, with $(gh)(u)=g(h(u))$, and $G$ acts
on $\widetilde S$ on the left. Its action on each fiber is free and
transitive. Consequently, any two lifts of the same map from a
connected domain differ by a unique deck transformation. We write
$gu$ for the action and $\Ad_d(g)=dgd^{-1}$ for conjugation in $G$.

\begin{lemma}\label{lem:existence}
The map $F=m\circ (p\times p):\widetilde S\times \widetilde S\to S$ has a continuous lift
$\mu:\widetilde S\times \widetilde S\to \widetilde S$.
\end{lemma}
\begin{proof}
$\widetilde S\times \widetilde S$ is path-connected, locally path-connected and simply
connected: coordinate loops contract in $\widetilde S$. Apply the covering-space
lifting theorem to $F$, with any chosen point in the fiber above
$F(u_0,v_0)$ as its prescribed value at $(u_0,v_0)$.
\end{proof}

\subsection{The defect and its normalization}

Fix a continuous multiplication lift $\mu$. The two bracketings
$\mu(\mu(u,v),w)$ and $\mu(u,\mu(v,w))$ have the same projection
because multiplication on $S$ is associative. The following lemma
records the deck transformation relating them, together with the
equivariance identities needed to correct the lift.

\begin{lemma}\label{lem:deck-data}
For every continuous lift $\mu:\widetilde S\times\widetilde S
\to\widetilde S$ of $m\circ (p\times p)$, there are unique
endomorphisms $\lambda_\mu,\rho_\mu:G\to G$ and a unique element
$d_\mu\in G$ such that
\begin{align}
 \mu(gu,v)&=\lambda_\mu(g)\mu(u,v),&
 \mu(u,gv)&=\rho_\mu(g)\mu(u,v),\label{eq:equivariance}\\
 \mu(\mu(u,v),w)&=d_\mu\,\mu(u,\mu(v,w)).\label{eq:defect}
\end{align}
We call $d_\mu$ the \emph{associativity defect} of $\mu$; the lift
is associative exactly when $d_\mu=1_G$. The endomorphisms
$\lambda_\mu$ and $\rho_\mu$ describe the deck action in the first
and second variables, respectively.
The images of $\lambda_\mu$ and $\rho_\mu$ commute elementwise, and
\begin{align}
 \lambda_\mu^2&=\Ad_{d_\mu}\circ \lambda_\mu,&
 \rho_\mu^2&=\Ad_{d_\mu^{-1}}\circ \rho_\mu,
 \label{eq:squares}\\
 d_\mu^2&=\lambda_\mu(d_\mu)d_\mu\rho_\mu(d_\mu).
 \label{eq:pentagon}
\end{align}
Every other lift is $\mu_c:=c\mu$ for a unique $c\in G$.
Its deck data depend on $c$ according to
\begin{equation}\label{eq:change-deck-data}
 \begin{aligned}
 \lambda_{\mu_c}&=\Ad_c\circ \lambda_\mu,\\
 \rho_{\mu_c}&=\Ad_c\circ \rho_\mu,\\
 d_{\mu_c}&=c\lambda_\mu(c)d_\mu\rho_\mu(c)^{-1}c^{-1}.
 \end{aligned}
\end{equation}
In particular, the lift $\mu_c$ is associative if and only if
\begin{equation}\label{eq:correction}
 d_\mu=\lambda_\mu(c)^{-1}\rho_\mu(c).
\end{equation}
\end{lemma}
\begin{proof}
Fix a continuous lift $\mu$.
For any $g\in G$, the maps $\mu(gu,v)$ and $\mu(u,v)$ lift the same
map from the connected space $\widetilde S^{\,2}$. They differ by a
unique deck transformation $\lambda_\mu(g)$. The same argument gives
$\rho_\mu(g)$. Applying two deck transformations in succession shows
that $\lambda_\mu(gh)=\lambda_\mu(g)\lambda_\mu(h)$ and similarly for
$\rho_\mu$. Changing the two inputs independently gives
\[
 \mu(gu,hv)=\lambda_\mu(g)\rho_\mu(h)\mu(u,v)
           =\rho_\mu(h)\lambda_\mu(g)\mu(u,v).
\]
Since the deck action is free,
$\lambda_\mu(g)\rho_\mu(h)=\rho_\mu(h)\lambda_\mu(g)$.
The two bracketings lift the same map on connected $\widetilde S^{\,3}$,
which gives the unique $d_\mu$ in \eqref{eq:defect}.

Set $L=\mu(\mu(u,v),w)$ and $R=\mu(u,\mu(v,w))$, so $L=d_\mu R$.
Replacing $u$ by $gu$ in this equality gives
$\lambda_\mu^2(g)L=d_\mu\lambda_\mu(g)R$, hence
$\lambda_\mu^2(g)=d_\mu\lambda_\mu(g)d_\mu^{-1}$.
Replacing $w$ by $gw$ gives
$\rho_\mu(g)L=d_\mu\rho_\mu^2(g)R$, hence
$\rho_\mu^2(g)=d_\mu^{-1}\rho_\mu(g)d_\mu$.
These are \eqref{eq:squares}.

To prove \eqref{eq:pentagon}, let $u,v,w,z\in\widetilde S$ be arbitrary.
The associativity-defect identity
\[
 \mu(\mu(u,v),w)=d_\mu\,\mu(u,\mu(v,w))
\]
gives two expansions:
\[
\begin{aligned}
 \mu(\mu(\mu(u,v),w),z)
 &=d_\mu\,\mu(\mu(u,v),\mu(w,z))\\
 &=d_\mu^2\,\mu(u,\mu(v,\mu(w,z))),
\end{aligned}
\]
whereas successive changes inside the first input, then outside,
then inside the second input give
\[
\begin{aligned}
 \mu(\mu(\mu(u,v),w),z)
 &=\lambda_\mu(d_\mu)\,\mu(\mu(u,\mu(v,w)),z)\\
 &=\lambda_\mu(d_\mu)d_\mu\,\mu(u,\mu(\mu(v,w),z))\\
 &=\lambda_\mu(d_\mu)d_\mu\rho_\mu(d_\mu)\,
   \mu(u,\mu(v,\mu(w,z))).
\end{aligned}
\]
Freeness of the deck action gives \eqref{eq:pentagon}.

Finally, all lifts differ from $\mu$ by a single deck transformation.
For $\mu_c=c\mu$, equivariance gives
\begin{align*}
 \mu_c(gu,v)&=c\lambda_\mu(g)c^{-1}\mu_c(u,v),\\
 \mu_c(u,gv)&=c\rho_\mu(g)c^{-1}\mu_c(u,v),
\end{align*}
proving the first two formulas in \eqref{eq:change-deck-data}.
For the two bracketings, direct calculation yields
\begin{align*}
 \mu_c(\mu_c(u,v),w)
 &=c\lambda_\mu(c)d_\mu\,\mu(u,\mu(v,w)),\\
 \mu_c(u,\mu_c(v,w))
 &=c\rho_\mu(c)\,\mu(u,\mu(v,w)).
\end{align*}
The corresponding defect element is
\[
 d_{\mu_c}=c\lambda_\mu(c)d_\mu\rho_\mu(c)^{-1}c^{-1},
\]
which proves the third formula in \eqref{eq:change-deck-data}.
The two bracketings agree exactly when
$\lambda_\mu(c)d_\mu=\rho_\mu(c)$, that is, when
\eqref{eq:correction} holds.
\end{proof}

The next lemma gives an explicit correction in the idempotent case.

\begin{lemma}[Normalization]\label{lem:normalize}
Let $\mu$ be a continuous multiplication lift. If either
$\lambda_\mu$ or $\rho_\mu$ is idempotent, then $\mu_c=c\mu$ is
associative, where
\[
 c:=\rho_\mu(d_\mu)\lambda_\mu(d_\mu)^{-1}.
\]
Moreover, if $\lambda_\mu$ is idempotent, then
$\lambda_{\mu_c}=\lambda_\mu$; if $\rho_\mu$ is idempotent,
then $\rho_{\mu_c}=\rho_\mu$.
\end{lemma}
\begin{proof}
Assume that either
$\lambda_\mu$ or $\rho_\mu$ is idempotent.
The elements $\lambda_\mu(d_\mu)$ and $\rho_\mu(d_\mu)$ commute,
since the images of $\lambda_\mu$ and $\rho_\mu$ commute elementwise.
If $\lambda_\mu$ is idempotent, \eqref{eq:squares} implies that
$d_\mu$ commutes with $\lambda_\mu(d_\mu)$.
If $\rho_\mu$ is idempotent, the same identities imply that
$d_\mu$ commutes with $\rho_\mu(d_\mu)$.

The pentagon identity reads
$d_\mu^2=\lambda_\mu(d_\mu)d_\mu\rho_\mu(d_\mu)$.
If $\lambda_\mu(d_\mu)$ commutes with $d_\mu$, this becomes
$d_\mu^2=d_\mu\lambda_\mu(d_\mu)\rho_\mu(d_\mu)$;
if $\rho_\mu(d_\mu)$ commutes with $d_\mu$, it becomes
$d_\mu^2=\lambda_\mu(d_\mu)\rho_\mu(d_\mu)d_\mu$.
Cancellation gives the same factorization in either case.

Thus
\begin{equation}\label{eq:factor-defect}
 d_\mu=\lambda_\mu(d_\mu)\rho_\mu(d_\mu).
\end{equation}
Consequently, $d_\mu$ commutes with both $\lambda_\mu(d_\mu)$
and $\rho_\mu(d_\mu)$.
Evaluating \eqref{eq:squares} at $d_\mu$ gives
\[
 \lambda_\mu^2(d_\mu)=\lambda_\mu(d_\mu),\qquad
 \rho_\mu^2(d_\mu)=\rho_\mu(d_\mu).
\]
Applying $\lambda_\mu$ and $\rho_\mu$ to
\eqref{eq:factor-defect} now gives
\[
 \lambda_\mu\bigl(\rho_\mu(d_\mu)\bigr)=1,\qquad
 \rho_\mu\bigl(\lambda_\mu(d_\mu)\bigr)=1.
\]
Thus, for $c:=\rho_\mu(d_\mu)\lambda_\mu(d_\mu)^{-1}$,
\[
 \lambda_\mu(c)=\lambda_\mu(d_\mu)^{-1},\qquad
 \rho_\mu(c)=\rho_\mu(d_\mu),
\]
and hence
\[
 \lambda_\mu(c)^{-1}\rho_\mu(c)
 =\lambda_\mu(d_\mu)\rho_\mu(d_\mu)
 =d_\mu.
\]
By \eqref{eq:correction}, $\mu_c$ is associative.

Suppose now that $\lambda_\mu$ is idempotent, and put
$H:=\lambda_\mu(G)$. By \eqref{eq:squares}, $d_\mu$
centralizes $H$. Applying $\lambda_\mu$ shows that
$\lambda_\mu(d_\mu)$ also centralizes $H$, since
$\lambda_\mu$ acts as the identity on $H$. Moreover,
$\rho_\mu(d_\mu)$ centralizes $H$ because the images of
$\lambda_\mu$ and $\rho_\mu$ commute elementwise.
Thus $c$ centralizes $H$, and \eqref{eq:change-deck-data} gives
\[
 \lambda_{\mu_c}=\Ad_c\circ \lambda_\mu=\lambda_\mu.
\]
Similarly, if $\rho_\mu$ is idempotent, then $d_\mu$ and
$\rho_\mu(d_\mu)$ centralize $\rho_\mu(G)$ by
\eqref{eq:squares}, while $\lambda_\mu(d_\mu)$
centralizes it by commutativity of the two images. Hence
$c$ centralizes $\rho_\mu(G)$ and
$\rho_{\mu_c}=\rho_\mu$.
\end{proof}

\subsection{The deck-endomorphism semigroup}

\begin{theorem}[Deck-endomorphism criterion for associative lifts]
\label{thm:deck-idempotents}
With the notation and hypotheses of this section, let
$m:S\times S\to S$ denote multiplication and put
\[
 \mathcal L_m(p)
 :=\{\lambda_\mu:\mu:\widetilde S\times\widetilde S
 \longrightarrow\widetilde S
 \text{ is a continuous lift of }m\}
 \subseteq\operatorname{End}(G).
\]
Then $\mathcal L_m(p)$ is a nonempty subsemigroup of
$\operatorname{End}(G)$ under composition. Moreover,
\begin{equation}\label{eq:realization-of-idempotents}
 E(\mathcal L_m(p))
 =\{\lambda_\nu:\nu\text{ is a continuous associative lift of }m\}.
\end{equation}
In particular, $m$ admits a continuous associative lift if and
only if $E(\mathcal L_m(p))\ne\varnothing$.
Equivalently, for any fixed continuous multiplication lift $\mu$,
this holds if and only if there exists $c\in G$ such that
\[
 (\Ad_c\circ \lambda_\mu)^2=\Ad_c\circ \lambda_\mu.
\]
\end{theorem}
\begin{proof}
Fix a continuous multiplication lift $\mu$, which exists by
Lemma~\ref{lem:existence}. By Lemma~\ref{lem:deck-data}, every
other lift is $g\mu$ for a unique $g\in G$, and
\[
 \mathcal L_m(p)=\{\Ad_g\circ \lambda_\mu:g\in G\}.
\]
Writing $d=d_\mu$, equation~\eqref{eq:squares} gives
$\lambda_\mu^2=\Ad_d\circ \lambda_\mu$. Hence, for $g,h\in G$,
\[
 (\Ad_g\circ \lambda_\mu)\circ (\Ad_h\circ \lambda_\mu)
 =\Ad_{g\lambda_\mu(h)d}\circ \lambda_\mu
 \in\mathcal L_m(p).
\]
Thus $\mathcal L_m(p)$ is a nonempty semigroup under composition.

If $\nu$ is associative, then $d_\nu=1_G$, and
\eqref{eq:squares} implies $\lambda_\nu^2=\lambda_\nu$.
Conversely, suppose that $e\in E(\mathcal L_m(p))$.
Choose a lift $\mu$ with $\lambda_\mu=e$. By
Lemma~\ref{lem:normalize}, the lift $\nu=c\mu$, where
\[
 c=\rho_\mu(d_\mu)\lambda_\mu(d_\mu)^{-1},
\]
is associative and satisfies $\lambda_\nu=\lambda_\mu=e$.
This proves~\eqref{eq:realization-of-idempotents} and the
existence criterion. The final equivalence follows from the
inner-orbit description of $\mathcal L_m(p)$ above.
\end{proof}

\begin{remark}[Non-uniqueness of multiplication lifts]
\label{rem:lift-nonuniqueness}
It is important to distinguish a deck endomorphism from a lift
inducing it. Write $\operatorname{Lift}(m,p)$ for the set of
continuous multiplication lifts. The assignment
\[
 \mathcal D:\operatorname{Lift}(m,p)\longrightarrow\mathcal L_m(p),
 \qquad \mathcal D(\mu)=\lambda_\mu,
\]
is surjective by definition, but in general not injective.
Fix a lift $\mu$. Every other lift is uniquely $b\mu$ for
some $b\in G$, and \eqref{eq:change-deck-data} shows that
\[
 \lambda_{b\mu}=\lambda_\mu
 \quad\Longleftrightarrow\quad
 b\in C_G\bigl(\lambda_\mu(G)\bigr),
\]
where $C_G(H)$ denotes the centralizer of $H$ in $G$.
Consequently,
\[
 \mathcal D^{-1}(\lambda_\mu)
 =\{b\mu:b\in C_G(\lambda_\mu(G))\}.
\]
In particular, if $G$ is abelian, all continuous multiplication
lifts induce the same endomorphism, although the lifts are
distinct when $G$ is nontrivial. The following example shows
that even an idempotent deck endomorphism may be induced by
non-associative lifts. Thus \eqref{eq:realization-of-idempotents}
asserts the existence of at least one associative representative
for every idempotent, not that all its representatives are associative.
\end{remark}

\begin{example}[A non-associative lift]\label{ex:left-zero-lifts}
Let $S=\RR/\ZZ$ with multiplication $x\cdot y=x$.
Every element of $S$ is idempotent (since
$x\cdot x=x$ for all $x\in S$).
Its universal covering is $p:\RR\to S$, $p(u)=[u]$.
For each $n\in\ZZ$, the map
\[
 \mu_n(u,v)=u+n
\]
is a continuous lift of the multiplication on $S$.
However,
\[
 \mu_n(\mu_n(u,v),w)=u+2n,
 \qquad
 \mu_n(u,\mu_n(v,w))=u+n.
\]
Thus $\mu_n$ is associative if and only if $n=0$.
This shows that even for a finite CW semigroup an arbitrary
continuous lift of multiplication need not be associative.

In the notation of Lemma~\ref{lem:deck-data}, identify $G$ with
$\ZZ$ acting by integer translations. Since
\[
 \mu_n(u+k,v)=\mu_n(u,v)+k\qquad(k,n\in\ZZ),
\]
we have $\lambda_{\mu_n}=\id_{\ZZ}$ for every $n\in\ZZ$.
Thus different lifts may induce the same endomorphism $\lambda_\mu$.
Moreover, $\lambda_{\mu_n}$ is always idempotent, although only
$\mu_0$ is associative: idempotence of the associated endomorphism
does not imply associativity of that particular lift.
\end{example}

\section{The Deck-Endomorphism Model and Proof of Theorem A}
\label{sec:proof-A}

Let $S$, $p:\widetilde S\to S$, and $G=\Deck(p)$ be as in
Section~\ref{sec:defect}. We identify the intrinsically defined
fundamental semigroup of Definition~\ref{def:fundamental-semigroup}
with a semigroup of deck endomorphisms, and then prove
Theorem~\ref{thm:main}. The change-of-basepoint
and covering-space facts used below are standard; see
\cite[Sections~1.1 and~1.3]{Hatcher}.

\subsection{The deck-endomorphism model}
\label{subsec:deck-endomorphism-semigroup}

Extending the notation used in Definition~\ref{def:fundamental-semigroup},
for a self-map $f:X\to X$ of a path-connected space and a path
$\alpha:x_0\to f(x_0)$, put
\[
 \theta_{f,\alpha}:=C_\alpha\circ f_*
 \in\operatorname{End}(\pi_1(X,x_0)).
\]
For $f=R_a$ and $\alpha:a\to a^2$, this agrees with the abbreviated
notation of Definition~\ref{def:fundamental-semigroup}:
\[
 \theta_{R_a,\alpha}=\theta_\alpha.
\]
For general self-maps and left translations we retain the map
subscript.
Here $C_\alpha([\ell])=[\alpha*\ell*\overline\alpha]$, with paths
concatenated from left to right. If $\beta$ has the same endpoints
as $\alpha$, the change-of-path formula is
\begin{equation}\label{eq:change-connecting-path}
 \theta_{f,\beta}=\Ad_b\circ \theta_{f,\alpha},
 \qquad b=[\beta*\overline\alpha]\in\pi_1(X,x_0).
\end{equation}
Thus changing the connecting path amounts to postcomposition by an
inner automorphism, which we call an \emph{output inner automorphism}.

Fix $a\in S$ and $u\in p^{-1}(a)$. The endpoint isomorphism
\begin{equation}\label{eq:endpoint-isomorphism}
 \Phi_u:\pi_1(S,a)\xrightarrow{\cong}G,
 \qquad [\ell]\longmapsto g_\ell,
\end{equation}
is characterized by the condition that the lift of $\ell$ starting
at $u$ ends at $g_\ell u$.

With our left-to-right convention for concatenating loops and the
left action of $G$, one has $g_{\ell_1*\ell_2}=g_{\ell_1}g_{\ell_2}$;
thus $\Phi_u$ is a homomorphism, rather than an antihomomorphism.

If $\widetilde f:\widetilde S\to\widetilde S$
lifts $f:S\to S$, uniqueness of lifts gives a homomorphism
$\varphi_{\widetilde f}:G\to G$ characterized by
$\widetilde f(gv)=\varphi_{\widetilde f}(g)\widetilde f(v)$.
Choose a path $\widetilde\alpha:u\to\widetilde f(u)$ and put
$\alpha=p\circ \widetilde\alpha$. Then
\begin{equation}\label{eq:covering-naturality}
 \varphi_{\widetilde f}
 =\Phi_u\circ \theta_{f,\alpha}\circ \Phi_u^{-1}.
\end{equation}
Indeed, if $\ell$ lifts from $u$ to $gu$, the loop
$\alpha*(f\circ \ell)*\overline\alpha$ lifts from $u$ to
$\varphi_{\widetilde f}(g)u$: its three pieces are
$\widetilde\alpha$, the image under $\widetilde f$ of the lift of
$\ell$, and $\varphi_{\widetilde f}(g)\overline{\widetilde\alpha}$.
The endpoint description of $\Phi_u$ therefore gives
\eqref{eq:covering-naturality}.

\subsection{Proof of Theorem A}

\begin{proof}[Proof of Theorem~\ref{thm:main}]
Fix $a\in S$ and $u_0\in p^{-1}(a)$. We work with $\Lambda(S,a)$
as a representative of $\Lambda(S)$.

\emph{Identification of the two models.}
We first show that conjugation by the endpoint isomorphism
$\Phi_{u_0}$ identifies $\Lambda(S,a)$ with $\mathcal L_m(p)$.
For a continuous multiplication lift $\mu$, define
$\widetilde R_a(v):=\mu(v,u_0)$. This is a lift of $R_a\circ p$,
and \eqref{eq:equivariance} gives
\[
 \widetilde R_a(gv)=\lambda_\mu(g)\widetilde R_a(v).
\]
Choose a path $\widetilde\alpha:u_0\to\mu(u_0,u_0)$ and put
$\alpha=p\circ \widetilde\alpha:a\to a^2$. Applying
\eqref{eq:covering-naturality} to $\widetilde R_a$ yields
\begin{equation}\label{eq:translation-deck-identification}
 \varphi_{\widetilde R_a}=\lambda_\mu
 =\Phi_{u_0}\circ \theta_\alpha\circ \Phi_{u_0}^{-1}.
\end{equation}

Given any path $\gamma:a\to a^2$, lift it from $u_0$ and let
$v\in p^{-1}(a^2)$ be its endpoint. The lifting theorem used in
Lemma~\ref{lem:existence} provides a multiplication lift $\mu$ with
$\mu(u_0,u_0)=v$. Taking $\widetilde\alpha$ to be the lift of $\gamma$
in \eqref{eq:translation-deck-identification}, we obtain
\[
 \Phi_{u_0}\bigl(C_\gamma\circ (R_a)_*\bigr)\Phi_{u_0}^{-1}
 =\lambda_\mu\in\mathcal L_m(p).
\]
Conversely, for any $\lambda_\mu\in\mathcal L_m(p)$, choose a path
from $u_0$ to $\mu(u_0,u_0)$ in $\widetilde S$. Its projection is a path
$\gamma:a\to a^2$, and the same identity gives the reverse inclusion.
These two inclusions, together with the fact that conjugation by a
group isomorphism preserves composition, establish the semigroup
isomorphism
\begin{equation}\label{eq:intrinsic-identification}
 \begin{aligned}
 \Lambda(S,a)&\xrightarrow{\cong}\mathcal L_m(p),\\
 \theta_\alpha&\longmapsto\Phi_{u_0}\circ \theta_\alpha\circ \Phi_{u_0}^{-1}.
 \end{aligned}
\end{equation}

\emph{Part (a): the idempotent criterion.}
By Theorem~\ref{thm:deck-idempotents}, a continuous associative
multiplication lift exists if and only if
$E(\mathcal L_m(p))\ne\varnothing$. The semigroup
isomorphism~\eqref{eq:intrinsic-identification} identifies this
condition with $E(\Lambda(S,a))\ne\varnothing$, proving
part~\textup{(a)}. The same theorem also shows that each
individual idempotent of $\Lambda(S,a)$ is realized by an
associative lift.

\emph{Part (b): uniqueness up to isomorphism over the base.}
Let $\mu$ and $\mu'$ be continuous associative lifts.
For $\mu$, we have $d_\mu=1_G$, and hence
\begin{equation}\label{eq:compare-idempotence}
 \lambda_\mu^2=\lambda_\mu,\qquad
 \rho_\mu^2=\rho_\mu.
\end{equation}
By Lemma~\ref{lem:deck-data}, there is a unique $c\in G$ such that
$\mu'=c\mu$. Since $\mu'$ is associative, the correction equation
\eqref{eq:correction} yields
\[
 \lambda_\mu(c)=\rho_\mu(c)=:h.
\]
Idempotence gives $\lambda_\mu(h)=h=\rho_\mu(h)$.
Set $g:=ch^{-2}\in G$. Then
\[
 \lambda_\mu(g)=\lambda_\mu(c)\lambda_\mu(h)^{-2}=h^{-1},
 \qquad
 \rho_\mu(g)=\rho_\mu(c)\rho_\mu(h)^{-2}=h^{-1}.
\]
Using \eqref{eq:equivariance}, we conclude that
\begin{equation}\label{eq:conjugacy}
 \begin{aligned}
 \mu'(gu,gv)
 &=c\,\mu(gu,gv)\\
 &=c\,\lambda_\mu(g)\rho_\mu(g)\mu(u,v)\\
 &=ch^{-2}\mu(u,v)=g\mu(u,v)
 \qquad(u,v\in\widetilde S).
 \end{aligned}
\end{equation}
Thus $g$ is the required topological semigroup isomorphism over $S$,
which completes part~\textup{(b)}.
\end{proof}

\begin{remark}\label{rem:left-translation-deck}
For any multiplication lift $\mu$, fix $a\in S$ and
$u_0\in p^{-1}(a)$, and let $\alpha$ be the projection of a path
from $u_0$ to $\mu(u_0,u_0)$. Applying the same argument to
$L_a(x)=ax$ and its lift $\widetilde L_a(v)=\mu(u_0,v)$ gives
\begin{equation}\label{eq:left-translation-deck-identification}
 \varphi_{\widetilde L_a}=\rho_\mu
 =\Phi_{u_0}\circ \theta_{L_a,\alpha}\circ \Phi_{u_0}^{-1},
 \qquad \theta_{L_a,\alpha}=C_\alpha\circ (L_a)_*.
\end{equation}
Since $C_\alpha$ and $\Phi_{u_0}$ are isomorphisms,
$\lambda_\mu$ is injective whenever $(R_a)_*$ is injective, and
$\rho_\mu$ is injective whenever $(L_a)_*$ is injective.
\end{remark}

\begin{remark}[Realization versus classification]
\label{rem:associative-lifts-idempotents}
Theorem~\ref{thm:deck-idempotents} and the semigroup
isomorphism~\eqref{eq:intrinsic-identification} show that every
$\alpha\in E(\Lambda(S,a))$ is realized by a continuous
associative lift $\nu$ such that, for $u_0\in p^{-1}(a)$,
\[
 \lambda_\nu=\Phi_{u_0}\circ \alpha\circ \Phi_{u_0}^{-1}.
\]
This is a realization statement, not a classification of
associative lifts by idempotents. As noted in
Remark~\ref{rem:lift-nonuniqueness} and
Example~\ref{ex:left-zero-lifts}, a lift inducing an
idempotent need not itself be associative. Moreover,
Theorem~\ref{thm:main}\textup{(b)} shows that all associative
lifts are isomorphic over $S$.
\end{remark}

\section{Proofs of Corollaries~\ref{cor:finite-inner},
\ref{cor:thompson-free} and~\ref{cor:CW}}

\label{sec:existence-criteria}

\subsection{Finite inner automorphism groups}

\begin{proof}[Proof of Corollary~\ref{cor:finite-inner}]
Fix $a\in S$ and a path $\alpha:a\to a^2$.
By the inner-orbit description \eqref{eq:lambda-inner-orbit},
\[
 \Lambda(S,a)=
 \bigl\{\varphi\circ \theta_\alpha:
 \varphi\in\operatorname{Inn}(\pi_1(S,a))\bigr\}.
\]
Hence
\[
 |\Lambda(S,a)|\leq
 |\operatorname{Inn}(\pi_1(S,a))|<\infty.
\]
Since $\Lambda(S,a)$ is a nonempty finite semigroup, it contains
an idempotent. Theorem~\ref{thm:main} then gives an associative
lift, unique up to isomorphism over $S$.
Finally, if $\pi_1(S,a)$ is abelian, its inner automorphism
group is trivial, whereas if $\pi_1(S,a)$ is finite, its inner
automorphism group is finite.
\end{proof}

\subsection{Fundamental groups without Thompson's group \texorpdfstring{$F$}{F}}

We recall the following purely group-theoretic consequence of
Freyd--Heller's Main Theorem~(T1), Main Corollary~(C1),
and Section~3 (L1) \cite{FreydHeller1993}.
Their result applies to an arbitrary group, not necessarily
a fundamental group.

\begin{lemma}[Freyd--Heller]\label{lem:freyd-heller-normalization}
Let $H$ be a group, and let $\varphi:H\to H$ be an endomorphism.
Suppose that, for some $d\in H$,
\[
 \varphi^2=\Ad_d\circ \varphi.
\]
Then there exists a canonical homomorphism from Thompson's
group $F$ to $H$,
\[
 \kappa_{\varphi,d}:F\longrightarrow H,
\]
defined by
\[
 \kappa_{\varphi,d}(x_n)=\varphi^n(d^{-1}),\qquad n\geq 0.
\]
Moreover, $\kappa_{\varphi,d}$ is injective if and only if
there is no $b\in H$ such that
\[
 (\Ad_b\circ \varphi)^2=\Ad_b\circ \varphi.
\]
In particular, if $H$ contains no subgroup isomorphic to
Thompson's group $F$, then there exists $b\in H$ such that
\[
 (\Ad_b\circ \varphi)^2=\Ad_b\circ \varphi.
\]
Equivalently, if the inner orbit of $\varphi$ contains no
idempotent endomorphism, then $F$ embeds in $H$.
\end{lemma}

\begin{proof}[Proof of Corollary~\ref{cor:thompson-free}]
Fix $a\in S$ and a path $\alpha:a\to a^2$.
By closure under composition and the inner-orbit
description~\eqref{eq:lambda-inner-orbit}, there exists
$d\in\pi_1(S,a)$ such that
\[
 \theta_\alpha^2
 =\Ad_d\circ \theta_\alpha.
\]
We apply Lemma~\ref{lem:freyd-heller-normalization} to the
group $\pi_1(S,a)$, with
$\varphi=\theta_\alpha$. The assumption on the
fundamental group therefore yields $b\in\pi_1(S,a)$ such that
\[
 (\Ad_b\circ \theta_\alpha)^2
 =\Ad_b\circ \theta_\alpha.
\]
By~\eqref{eq:lambda-inner-orbit},
$E(\Lambda(S,a))\ne\varnothing$.
Equivalently, the lemma yields the necessary condition
\begin{equation}\label{eq:freyd-heller-obstruction}
 E(\Lambda(S,a))=\varnothing
 \quad\Longrightarrow\quad
 F\hookrightarrow\pi_1(S,a).
\end{equation}
Theorem~\ref{thm:main} now supplies the associative lift and
its uniqueness up to isomorphism over $S$.

For the stated classes, recall that $F$ is torsion-free,
that $F'$ is infinite and nonabelian simple (so $F$ is not
solvable), and that every finite quotient of $F$ kills $F'$
(so $F$ is not residually finite). Thus torsion, solvable,
and residually finite groups cannot contain $F$.
Finally, a subgroup of a linear group
isomorphic to $F$ would be a finitely generated linear group,
hence residually finite by Mal'cev's theorem \cite{Malcev1940};
since $F$ is not residually finite, linear groups cannot contain
$F$ either (this case is also included in Freyd--Heller's Main
Corollary (C4)).
Finally, amenability is inherited by subgroups.
Since Thompson's group $F$ is nonamenable
\cite{OpenAI2026Nonamenable,KimLodha2026},
no amenable group can contain a subgroup isomorphic to $F$.
Thus the same conclusion holds for amenable fundamental groups.
\end{proof}

\subsection{Homotopy splitting for finite-dimensional CW complexes}

\begin{definition}
A continuous map $f:X\to X$ is a
\emph{homotopy idempotent} if
$f\sim f^2$, where $f^2=f\circ f$.
The map $f$ \emph{splits} if there exist a topological space $Y$ and
continuous maps $r:X\to Y$ and $i:Y\to X$ such that
\[
 f\sim i\circ r,\qquad r\circ i\sim\id_Y.
\]
The space $Y$ is called a \emph{splitting space}; it is not required
to be a subspace of $X$. The map $r$ is called a \emph{homotopy
retraction}, and $i$ is a \emph{homotopy section} of $r$.
The condition $r\circ i\sim\id_Y$ defines the homotopy retraction;
the additional condition $f\sim i\circ r$ makes it a splitting of
the specified map $f$. These conditions imply
$f\sim ir\sim i(ri)r\sim f^2$.
This is a condition on a map, and does not assert that $X$ contains
an idempotent point for any multiplication.
\end{definition}

\begin{theorem}[Hastings--Heller]\label{thm:HH}
Every homotopy idempotent on a finite-dimensional CW
complex (possibly disconnected) splits.
\end{theorem}
This is the classical splitting theorem of
Hastings--Heller~\cite{HastingsHeller}.
For a modern treatment of the splitting argument for homotopy
idempotents, see Dydak~\cite[Section~5]{Dydak}.

\subsection{Normalization from a homotopy splitting}
\label{subsec:normalization-splitting}

The standard naturality and homotopy identities for change of
basepoint, recalled in Section~\ref{sec:fundamental-semigroup}
as~\eqref{eq:standard-naturality} and~\eqref{eq:standard-homotopy},
will also be used in the normalization argument below.

The following lemma is a consequence of the splitting principle for
idempotents modulo inner automorphisms established by Freyd--Heller
\cite[Section~3, (L1), p.~98]{FreydHeller1993}.
For the reader's convenience, we give a direct proof using paths
and change-of-basepoint isomorphisms, avoiding the categorical
formulation of the general result.

\begin{lemma}\label{lem:splitting_group}
Let $f:X\to X$ be a continuous map on a path-connected space and
suppose that $f$ splits up to homotopy through a
path-connected space $Y$: there exist $r:X\to Y$ and $i:Y\to X$
with $f\sim ir$ and $ri\sim\id_Y$. Fix $x_0\in X$.
Then, for every path $\alpha:x_0\to f(x_0)$, there exists
$b\in\pi_1(X,x_0)$ such that
\[
 (\Ad_b\circ \theta_{f,\alpha})^2
 =\Ad_b\circ \theta_{f,\alpha}.
\]
\end{lemma}

\begin{proof}
Put $y_0=r(x_0)$, $x_1=i(y_0)$, and
$y_1=r(x_1)=ri(y_0)$. Choose a homotopy from $\id_Y$ to $ri$,
and let $\delta:y_0\to y_1$ be the path traced by $y_0$.
The homotopy change-of-basepoint identity~\eqref{eq:standard-homotopy} gives
\begin{equation}\label{eq:normalized-ri}
 C_\delta\circ (ri)_*=\id_{\pi_1(Y,y_0)}.
\end{equation}

Choose any path $\sigma_0:x_0\to x_1$. Then
$r\circ \sigma_0:y_0\to y_1$, and we set
\[
 q=[(r\circ \sigma_0)*\overline\delta]\in\pi_1(Y,y_0).
\]
Let $\eta$ be a loop at $y_0$ representing $q^{-1}$, and
replace $\sigma_0$ with the path
\[
 \sigma:=\sigma_0*(i\circ \eta):x_0\to x_1.
\]
This choice eliminates the discrepancy between the two paths in $Y$:
by \eqref{eq:normalized-ri},
\[
\begin{aligned}
 {}[(r\circ \sigma)*\overline\delta]
 &=q\, [\delta*(ri\circ \eta)*\overline\delta]\\
 &=q\, C_\delta\bigl((ri)_*([\eta])\bigr)\\
 &=q[\eta]=1.
\end{aligned}
\]
Thus $r\circ \sigma$ and $\delta$, both paths from $y_0$ to $y_1$,
are homotopic relative to their endpoints. By the naturality
identity~\eqref{eq:standard-naturality},
\[
 r_*\circ C_\sigma\circ i_*
 =C_{r\circ \sigma}\circ (ri)_*
 =C_\delta\circ (ri)_*
 =\id_{\pi_1(Y,y_0)}.
\]
Consequently, the endomorphism
$C_\sigma\circ i_*\circ r_*$ of $\pi_1(X,x_0)$ is
idempotent, since
\[
\begin{aligned}
 (C_\sigma\circ i_*\circ r_*)^2
 &=C_\sigma\circ i_*\circ
 (r_*\circ C_\sigma\circ i_*)\circ r_*\\
 &=C_\sigma\circ i_*\circ r_*.
\end{aligned}
\]

Finally, choose a homotopy from $f$ to $ir$ and let
$\tau:f(x_0)\to ir(x_0)=x_1$ be the path traced by $x_0$.
Applying the homotopy identity~\eqref{eq:standard-homotopy} to this homotopy yields
\[
 \theta_{f,\alpha}
 =C_{\alpha*\tau}\circ i_*\circ r_*.
\]
Since $\sigma$ and $\alpha*\tau$ have the same endpoints, the
change-of-path formula gives
\[
 \Ad_b\circ \theta_{f,\alpha}
 =C_\sigma\circ i_*\circ r_*,\qquad
 b=[\sigma*\overline{(\alpha*\tau)}]\in\pi_1(X,x_0).
\]
Hence $(\Ad_b\circ \theta_{f,\alpha})^2
=\Ad_b\circ \theta_{f,\alpha}$, as required.
\end{proof}

\begin{proof}[Proof of Corollary~\ref{cor:CW}]
Fix $a\in S$ and choose a path $\alpha:a\to a^2$.
Since $S$ is path-connected, right translation $R_a$ is
homotopic to $R_{a^2}=R_a^2$: if $\gamma$ is a path from
$a$ to $a^2$, the homotopy is given by $(x,t)\mapsto x\gamma(t)$.
Thus $R_a$ is a homotopy idempotent.

Since $S$ is a finite-dimensional CW complex, the
Hastings--Heller theorem (Theorem~\ref{thm:HH}) implies
that $R_a$ splits up to homotopy.

The splitting space $Y$ is path-connected: $r(S)$ is
path-connected, and the homotopy $ri\sim\id_Y$ joins every
$y\in Y$ to $ri(y)\in r(S)$.

By Lemma~\ref{lem:splitting_group}, there exists
$b\in\pi_1(S,a)$ such that
\[
 (\Ad_b\circ \theta_\alpha)^2
 =\Ad_b\circ \theta_\alpha.
\]
By the inner-orbit description~\eqref{eq:lambda-inner-orbit},
$\Ad_b\circ \theta_\alpha\in\Lambda(S,a)$.
Hence $E(\Lambda(S))\ne\varnothing$.
The conclusion follows from Theorem~\ref{thm:main}.
\end{proof}

\section{Proof of Theorem B}\label{sec:proof-B}

We use the free strict monoidal groupoid of Fiore and Leinster
\cite{FioreLeinster} and their identification of the automorphism
group of its generating object with Thompson's group $F$.
This construction is closely related to the theory of diagram groups
and Squier complexes developed by Guba and Sapir \cite{GubaSapir1997}.

Our contribution is not the underlying categorical construction,
but its application to the associative lifting problem for
topological semigroups. We show that the tensor-induced multiplication
on the classifying space gives rise to a connected countable CW
topological semigroup whose multiplication admits continuous lifts
to the universal cover, but no associative lift. The obstruction is
detected by the injectivity of the homomorphisms induced by left and
right translations on the nonabelian fundamental group.

This provides the counterexample asserted in
Theorem~\ref{thm:counterexample} and shows that the
finite-dimensional hypothesis of Corollary~\ref{cor:CW}
cannot be omitted.

\subsection{A fundamental-group obstruction}
\label{app:group-obstruction}

Let $S$ be a path-connected, locally path-connected and semilocally
simply connected topological semigroup and let $p:\widetilde S\to S$ be its
universal cover. The group $G$ acts on the left. We use the deck data
and the correction equation of Lemma~\ref{lem:deck-data}, whose proof
did not use finite-dimensionality.

\begin{proposition}
\label{app:prop:obstruction}
Let $S$ and $p:\widetilde S\to S$ be as above, and let $a\in S$. Suppose that
both left and right translation by $a$ induce injections
\begin{align*}
 (L_a)_*&:\pi_1(S,a)\longrightarrow\pi_1(S,a^2),
 & L_a(x)&=ax,\\
 (R_a)_*&:\pi_1(S,a)\longrightarrow\pi_1(S,a^2),
 & R_a(x)&=xa.
\end{align*}
If $\pi_1(S,a)$ is nonabelian, no continuous associative lift of the
multiplication of $S$ exists on $\widetilde S$.
\end{proposition}

\begin{proof}
Suppose that a compatible associative lift $\mu$ exists.
By Lemma~\ref{lem:deck-data}, associativity gives
$d_\mu=1$, and hence \eqref{eq:squares} implies
\[
 \lambda_\mu^2=\lambda_\mu,
 \qquad \rho_\mu^2=\rho_\mu.
\]
By \eqref{eq:translation-deck-identification} and
Remark~\ref{rem:left-translation-deck}, the assumed
injectivity of $(R_a)_*$ and $(L_a)_*$ implies that
both $\lambda_\mu$ and $\rho_\mu$ are injective.

An injective idempotent self-map of a set is necessarily the identity.
Therefore
\[
 \lambda_\mu=\rho_\mu=\id_G.
\]
But the images of $\lambda_\mu$ and $\rho_\mu$
commute elementwise by Lemma~\ref{lem:deck-data}.
Thus $G$ is abelian, contradicting
$G\cong\pi_1(S,a)$ and the hypothesis
that $\pi_1(S,a)$ is nonabelian.
\end{proof}

\begin{remark}
An alternative proof of Proposition~\ref{app:prop:obstruction}
can be obtained from the idempotent criterion of
Theorem~\ref{thm:main} and the inner-orbit description
of the fundamental semigroup.
The proof given above is more direct and makes
the obstruction transparent in terms of
deck transformations.
\end{remark}

\subsection{A classifying-space construction}
\label{app:Fiore-Leinster}

Following Fiore--Leinster \cite{FioreLeinster}, let $\mathcal A$ be
the free \emph{strict monoidal category} generated by an object $A$
and a specified isomorphism
\[
\alpha:A\otimes A\xrightarrow{\cong}A.
\]
Recall that a monoidal category has a tensor product $\otimes$
on objects and morphisms; ``strict'' means that associativity
holds exactly, rather than up to isomorphism. A \emph{groupoid}
is a category in which every morphism is invertible.

The category $\mathcal A$ is a countable groupoid: its morphisms
are finite expressions obtained from $\alpha$, $\alpha^{-1}$,
identities, composition, and tensor products. Fiore and
Leinster prove that
\begin{equation}\label{app:eq:Fiore-Leinster}
 \operatorname{Aut}_{\mathcal A}(A)\cong F,
\end{equation}
where $F$ is Thompson's nonabelian group
\cite[Theorem~1.1 and Section~3]{FioreLeinster}.

Let $\mathcal C$ be the full subcategory of $\mathcal A$
whose objects are $A^{\otimes n}$, $n\geq1$; thus we retain
all morphisms between these objects. It is a countable
connected groupoid, since $\alpha$ identifies
$A^{\otimes2}$ with $A$, and hence all positive tensor
powers are isomorphic. The tensor product restricts to a
strictly associative functor
$\mathcal C\times\mathcal C\to\mathcal C$.

Put
\[
 S_\infty=B\mathcal C=|N\mathcal C|.
\]
Here $N\mathcal C$ is the \emph{nerve} of $\mathcal C$:
its $k$-simplices are strings of $k$ composable morphisms.
The space $B\mathcal C$ is their geometric realization,
obtained by gluing the corresponding simplices.
In particular, each object $X$ determines a vertex
$v_X\in S_\infty$.

\begin{proof}[Proof of Theorem~\ref{thm:counterexample}]
Since $\mathcal C$ has countably many morphisms, $S_\infty$ is a
Hausdorff CW complex with countably many cells
\cite[Theorem~14.1]{May}. It is path-connected because
$\mathcal C$ is connected.
Because the nerve is a countable simplicial set, the realization
product theorem \cite[Theorem~14.3 and Remark~14.4]{May}
applies here to finite products with the ordinary product topology.
Hence the tensor functor induces a continuous multiplication on
$S_\infty$. Functoriality of realization and strict associativity
of the tensor product make this multiplication strictly associative,
not merely associative up to homotopy.

It satisfies
\[
 v_Xv_Y=v_{X\otimes Y}.
\]

Choose the point $a=v_A\in S_\infty$, namely the vertex
corresponding to the object $A$. Thus
$a^2=v_{A\otimes A}$; note that $a^2\ne a$, since
$A\otimes A$ and $A$ are distinct objects, although
isomorphic. The standard fundamental-group calculation
for classifying spaces of connected groupoids
\cite[Section~6]{Brown1987} gives
\[
 \pi_1(S_\infty,a)
 \cong\operatorname{Aut}_{\mathcal C}(A)
 \cong F,
\]
up to the harmless opposite-group convention.

The right and left translations by $a$ induce,
respectively, the homomorphisms
\[
 f\longmapsto f\otimes\id_A,
 \qquad
 f\longmapsto\id_A\otimes f
\]
from $\operatorname{Aut}_{\mathcal C}(A)$ to
$\operatorname{Aut}_{\mathcal C}(A\otimes A)$.

In the dyadic piecewise-linear model of Fiore--Leinster
\cite[Sections~2--3]{FioreLeinster}, these operations, after
identifying $A\otimes A$ with $A$, are rescaled copies of the
original homeomorphism supported on the left and right halves
of the unit interval, respectively. Restricting to the
corresponding half recovers the original homeomorphism.
Thus both induced homomorphisms are injective.

Since $F$ is nonabelian, Proposition~\ref{app:prop:obstruction}
excludes any associative lift of the multiplication, so
$E(\Lambda(S_\infty))=\varnothing$ by Theorem~\ref{thm:main}.
Continuous lifts exist by Lemma~\ref{lem:existence}.
Finally, repeating a
nonidentity automorphism of $A$ gives a nondegenerate
simplex in every dimension, so $S_\infty$ has cells
in arbitrarily high dimensions.
\end{proof}

\begin{remark}[Asphericity and infinite dimension]\label{rem:aspherical-F}
The infinite dimensionality of $S_\infty=B\mathcal C$ also follows
from its asphericity. Indeed, since $\mathcal C$ is a connected
groupoid with automorphism group isomorphic to Thompson's group $F$,
the space $S_\infty$ is a $K(F,1)$.

Suppose that $S_\infty$ had finite CW dimension $n$. Since $F$
contains a subgroup isomorphic to $\mathbb Z^{n+1}$, the
corresponding covering space $Y$ would be an $n$-dimensional
$K(\mathbb Z^{n+1},1)$. Hence $Y$ would be homotopy equivalent
to the $(n+1)$-dimensional torus $T^{n+1}$, giving
\[
H_{n+1}(Y;\mathbb Z)
\cong H_{n+1}(T^{n+1};\mathbb Z)
\cong \mathbb Z.
\]
This contradicts the fact that an $n$-dimensional CW complex has
vanishing homology in degrees greater than $n$. Thus $S_\infty$
is necessarily infinite-dimensional.
This agrees with the theorem of
Brown--Geoghegan \cite{BrownGeoghegan1984} that $F$ is a
torsion-free $\mathrm{FP}_\infty$ group of infinite
cohomological dimension.
\end{remark}

\section*{Declaration on the use of AI tools}
The ideas, results and proofs presented in this paper are entirely
due to the authors. Generative AI tools, namely Claude (Anthropic)
and ChatGPT (OpenAI), were used to assist in checking the formal
correctness of the arguments and in improving the presentation and
language of the manuscript. All outputs of these tools were
critically reviewed and verified by the authors, who take full
responsibility for the content of this paper.

\section*{Acknowledgements}
Andrea Loi acknowledges the support of the Fondazione di Sardegna
through the project ProBiki, and of INdAM--GNSAGA
(Gruppo Nazionale per le Strutture Algebriche, Geometriche e loro
Applicazioni).

\noindent Stefano Bonzio acknowledges the support of the Fondazione di Sardegna
through the project LoMEA, and of INdAM--GNSAGA
(Gruppo Nazionale per le Strutture Algebriche, Geometriche e loro
Applicazioni).

\bigskip
\noindent\textsc{Stefano Bonzio}\\
Department of Mathematics and Computer Science,\\
University of Cagliari, Via Ospedale 72, 09124 Cagliari, Italy\\
\textit{Email address:} \href{mailto:stefano.bonzio@unica.it}{\texttt{stefano.bonzio@unica.it}}

\medskip
\noindent\textsc{Andrea Loi}\\
Department of Mathematics and Computer Science,\\
University of Cagliari, Via Ospedale 72, 09124 Cagliari, Italy\\
\textit{Email address:} \href{mailto:loi@unica.it}{\texttt{loi@unica.it}}

\end{document}